\documentclass[a4paper]{article}   
\usepackage{float}                   
\usepackage{amsmath}
\usepackage{amssymb}
\usepackage{geometry}
\usepackage[hidelinks]{hyperref}
\usepackage[T1]{fontenc}
\usepackage{amsthm}
\usepackage{dsfont}
\usepackage{mathrsfs}

\usepackage{setspace}

\usepackage[
    backend=biber,
    style=alphabetic,
    sorting=nyt
]{biblatex}

\newcommand{\nbN}{\mathbb{N}}
\newcommand{\nbR}{\mathbb{R}}

\newcommand{\Scal}{\mathcal{S}}

\newcommand{\Ric}{\operatorname{Ric}}
\newcommand{\Rgot}{\mathfrak{R}}
\newcommand{\Blap}{\overline{\Delta}}

\newcommand{\X}{\mathscr{X}}
\newcommand{\T}{\mathscr{T}^0_2}
\newcommand{\TT}{\mathscr{T}_{TT}}
\newcommand{\supp}{\operatorname{supp}}

\newcommand{\TtG}{\mathscr{T}^G_{TT}}
\newcommand{\Span}{\operatorname{Span}}

\newcommand{\ti}{\mathfrak{T}_i}
\renewcommand{\t}{\mathfrak{T}}

\newtheorem{theo}{Theorem}[section]
\newtheorem{coro}[theo]{Corollary}
\newtheorem{prop}[theo]{Proposition}
\newtheorem{lem}[theo]{Lemma}
\newtheorem{defi}[theo]{Definition}
\newtheorem*{theoAen}{Theorem A}

\newcommand{\rem}[1]{\noindent \textit{Remark:} #1

\medskip}
\title{\textsc{{Instability of Böhm's Einstein metrics}}}

\author{Guillaume Verger%
  \thanks{Département de mathématiques, ENS Paris-Saclay, Gif-sur-Yvette, France.
  \\ Email: \href{mailto:guillaume.verger@ens-paris-saclay.fr}
  \texttt{guillaume.verger@ens-paris-saclay.fr}}}
\date{}
\begin{document}

\maketitle

\begin{abstract}
    Böhm metrics on $S^{k+1}\times S^l$ and $S^{k+l+1}$ ($k,l\geqslant 2$, $k+l\leqslant 8$) occur in sequences of Einstein metrics that converge to a cone. We prove that, along such a sequence, the number of negative eigenvalues of the Lichnerowicz Laplacian acting on transverse-traceless tensors tends to infinity: these metrics become increasingly unstable. This result gives a partial answer to a conjecture of Gibbons, Hartnoll and Pope \cite{lichbohm} concerning the instability of the generalised black hole spacetimes built from Böhm metrics.
\end{abstract}

\section{Introduction}

A Riemannian metric $g$ on a manifold $M$ is said to be Einstein if it satisfies $$\Ric(g)=\Lambda g$$ for some real constant $\Lambda$. The first non-trivial Einstein metrics on spheres were discovered in the 1970s by Jensen on $S^{4m+3}$ \cite{Jensen} and by Bourguignon and Karcher on $S^{15}$ \cite{bourg}. These metrics are homogeneous, and Ziller proved that there exist no other homogeneous Einstein metrics on spheres \cite{ziller1982homogeneous}. For low-dimensional spaces, such as spheres and products of spheres with dimensions between 5 and 9, Böhm was the first to construct Einstein metrics \cite{bohm}. These cohomogeneity one metrics are the first examples of inhomogeneous Einstein metrics on spheres. In 2017, Foscolo--Haskins constructed new Einstein metrics on $S^6$ and $S^3\times S^3$ \cite{foscolo2017new}. More recently, Nienhaus--Wink \cite{s10} and Buttsworth--Hodgkinson \cite{buttsworth2026computationally} have proved, respectively, the existence of cohomogeneity one Einstein metrics on $S^{10}$ and $S^{12}$.

When we study Einstein metrics, one of the natural questions concerns their stability. Indeed, Einstein metrics on $M$ are the critical points of the Einstein--Hilbert action \cite{hilbert1915grundlagen}
$$\begin{aligned}
\Scal: \mathcal M_V &\longrightarrow \nbR\\
g &\longmapsto \int_M\operatorname{scal}_gdV_g,
\end{aligned}$$
where $\mathcal M_V$ denotes the set of Riemannian metrics on $M$ of fixed volume $V>0$, and $\operatorname{scal}_g$ the scalar curvature. Thus, studying the stability of these metrics means studying the nature of the critical points of $\Scal$, which depends on the direction in which the metric is perturbed. This leads to the notion of \textit{unstable directions}, that is, directions along which $\Scal'$ increases. The behaviour of $\Scal$ in some directions, such as those given by conformal deformations (see \cite{besse}, Chapter 4), is already well understood. More precisely, there are infinitely many such directions, and they are all unstable. Moreover, $\Scal$ is invariant under the action of diffeomorphisms. Hence, the relevant directions are those orthogonal to the conformal directions and diffeomorphisms: transverse-traceless (TT) tensors. This restriction is necessary in order to obtain a finite number of unstable directions: this number is infinite when all directions are considered, but finite for transverse-traceless perturbations of a smooth compact Einstein manifold. This number is known as the \textit{coindex}.

The variational study of the stability problem goes back to Muto \cite{muto1974einstein} and Koiso \cite{koiso1979second}. More recently, Kröncke wrote a survey of the stability problem of compact Einstein
manifolds and its applications to mathematical physics \cite{kroncke2016variational}.

Furthermore, there is another notion of stability, more closely related to general relativity: the stability of generalised black holes associated with these metrics \cite{gibbons2002gravitational}.

Regardless of the notion of stability, the Lichnerowicz Laplacian $\Delta_L$ provides a useful tool for studying the stability of Einstein metrics. Originally introduced in the 1960s to study perturbations of solutions to Einstein's equations in general relativity \cite{lich}, it appears naturally when one linearises the Ricci tensor. This is the reason why it is related to the first notion of stability. Regarding the second one, a criterion established in 2002 by Gibbons and Hartnoll \cite{gibbons2002gravitational} links the stability of generalised Schwarzschild--Tangherlini black holes to the smallest eigenvalue of the Lichnerowicz Laplacian.

Concerning Böhm metrics, the only stability results we are aware of concern $S^3\times S^2$, $S^3\times S^3$, and $S^5$ \cite{lichbohm}. Gibbons, Hartnoll and Pope showed that, on these spaces, the Lichnerowicz Laplacian for every Böhm metric admits at least one negative mode. They also conjectured that this result holds in higher dimensions. That paper focuses on the second notion of stability and provides little information about the coindex, \textit{i.e.} how unstable Böhm metrics are, in the sense of the first notion.

In this paper, we are interested in the number of unstable directions of Böhm metrics. More precisely, let $(g_i)$ be a sequence of Böhm metrics on $S^{k+1}\times S^l$ or $S^{k+l+1}$ ($k,l\geqslant 2$, $k+l\leqslant 8$) constructed in \cite{bohm}. We aim to answer the following question: how does the number of unstable directions evolve as we move along the sequence? The following result answers this question (see Section \ref{secMain} for the precise statement).

\begin{theoAen}
    If $(g_i)$ denotes a sequence of Böhm metrics on $S^{k+1}\times S^l$ or $S^{k+l+1}$, with $k,l\geqslant 2$, $k+l\leqslant 8$, and if $\nu_i$ denotes the number of negative eigenvalues of the Lichnerowicz Laplacian associated with $g_i$ (called the \textit{index}), then
$$\nu_i\underset{i\to+\infty}{\longrightarrow}{+\infty}.$$

\end{theoAen}

Since $\Lambda>0$, every negative eigenvalue of $\Delta_L$ lies below the instability bound $2\Lambda$. Therefore, the index is less than or equal to the number of unstable directions (\textit{i.e.} the coindex). Hence, this quantity also tends to infinity as $i\to +\infty$.

This also provides a partial answer to the conjecture of Gibbons, Hartnoll and Pope \cite{lichbohm}: for all the products $S^{k+1}\times S^l$ and the spheres $S^{k+l+1}$ that admit Böhm metrics, generalised Schwarzschild--Tangherlini black holes are unstable, except possibly a finite number of them.

To prove Theorem A, we proceed as follows. We know that a sequence of Böhm metrics converges to the cone metric (see \cite{bohm}, Theorem 3.7). Therefore, we study the Lichnerowicz Laplacian associated with this metric. Kröncke proved that this operator is unbounded below in exactly these dimensions (see  \cite{kroncke2021spectra}, Corollary 1.6). We need something stronger: explicit unstable directions, which allow us to pass to the limit in the quadratic form and then apply the min-max theorem. Thus, we can find a subsequence of Böhm metrics such that the index tends to infinity, and then conclude the proof by contradiction.

This paper is organised as follows. In Section \ref{secBohmmetrics}, we recall some definitions and useful results about Böhm metrics. In Section \ref{secLichLap}, we define the Lichnerowicz Laplacian and give an explicit expression of this operator for Böhm metrics when we consider $G$-invariant TT tensors, where $G$ acts with cohomogeneity one. In this section, we also define two different notions of stability and explain the link between these notions and the spectrum of the Lichnerowicz Laplacian. The precise statements of our main results are given in Section \ref{secMain}. Then, since we know from \cite{bohm} that a sequence of Böhm metrics converges to the cone metric $g_c$ in the Gromov--Hausdorff distance (see Theorem \ref{convtocone}), in Section \ref{secConv}, we strengthen the convergence of these metrics when we stay far from the singular orbits. In Section \ref{secConemetric}, we study some properties of the Lichnerowicz Laplacian associated with the cone metric. In some sense, we prove that the index (\textit{i.e.} the number of negative eigenvalues) of this operator is infinite. Such a property is possible since the cone metric is not smooth at the singular orbits. Finally, in Section \ref{secClose}, we prove that the sequence of indices associated with a sequence of Böhm metrics diverges to infinity.

\paragraph{Acknowledgements.} I would like to thank Professor Jason D. Lotay for his support, without which this paper would not have been possible. I am also grateful to Qiu Shi Wang for his advice and for his feedback on the manuscript. I also thank the Mathematical Institute, Oxford University, for providing a welcoming environment during this project. Finally, I thank Professors Christoph Böhm and Matthias Wink, who kindly answered my questions about their articles.

\section{Böhm metrics on low-dimensional spaces}
\label{secBohmmetrics}

Let $k,l$ be integers such that $k,l\geqslant 2$ and $k+l\leqslant 8$. We recall the definition of Böhm metrics on $S^{k+1}\times S^l$ and $S^{k+l+1}$ \cite{bohm}. We denote by $G$ the group $SO(k+1)\times SO(l+1)$. It acts with cohomogeneity one on $S^{k+1}\times S^l$ and $S^{k+l+1}$.

Böhm seeks cohomogeneity one metrics of the following form
$$g(t)=dt^2+f(t)^2g^k+h(t)^2g^l,$$ where $g^k$ and $g^l$ denote the round metrics on $S^k$ and $S^l$, respectively.

He imposes the following boundary conditions on $f$ and $h$, which depend on whether the metric is defined on $S^{k+1}\times S^l$ or on $S^{k+l+1}$ ($'$ denotes $d/dt$):
\begin{itemize}
    \item On $S^{k+1}\times S^l$,

$$f(0)=0,\quad f'(0)=1,\quad h(0)=\bar h>0,\quad h'(0)=0,$$

$$f\left(\t\right)=0,\quad f'\left(\t\right)=-1,\quad h\left(\t\right)=\bar h,\quad h'\left(\t\right)=0,$$
where $\t>0$ denotes the distance between the two singular orbits.

\item On $S^{k+l+1}$,

$$f(0)=0,\quad f'(0)=1,\quad h(0)=\bar h>0,\quad h'(0)=0,$$

$$f\left(\t\right)=\bar f>0,\quad f'\left(\t\right)=0,\quad h\left(\t\right)=0,\quad h'\left(\t\right)=-1.$$
\end{itemize}

Böhm studies the existence of such metrics. The Einstein equations  for these metrics are \begin{align}
    \label{ein1}
    \Lambda &=-\left(k\frac{f''}{f}+l\frac{h''}{h}\right),
    \\ \Lambda  &=-\frac{f''}{f}+(k-1)\frac{1-f'^2}{f^2}-l\frac{f'h'}{fh},
    \label{ein2}
    \\ \Lambda &=-\frac{h''}{h}+(l-1)\frac{1-h'^2}{h^2}-k\frac{f'h'}{fh},
        \label{ein3}
\end{align} where $\Lambda>0$ is the Einstein constant.

Böhm proved that there exist sequences of Einstein metrics of the above form on $S^{k+1}\times S^l$ and $S^{k+l+1}$. More precisely, he proved the following theorems.

\begin{theo}[Existence of Einstein metrics on $S^{k+1}\times S^l$, Theorem 3.4 \cite{bohm}]
\label{existprod}
For all $(k,l)\in \nbN^2$ such that $k,l\geqslant 2$ and $k+l\leqslant 8$, there exists a sequence of pairwise non-isometric Einstein metrics $(g_i)$ on $S^{k+1}\times S^l$ with Einstein constant $\Lambda=k+l$. Moreover, the sequence of initial conditions $(\bar h_i)$ is decreasing and converges to 0.
\end{theo}

\begin{theo}[Existence of Einstein metrics on $S^{k+l+1}$, Theorem 3.6 \cite{bohm}]
\label{existsph}
For all $(k,l)\in \nbN^2$ such that $k,l\geqslant 2$ and $k+l\leqslant 8$, there exists a sequence of pairwise non-isometric Einstein metrics $(g_i)$ on $S^{k+l+1}$ with Einstein constant $\Lambda=k+l$. Moreover, if we denote by $(\bar f_i, \bar h_i)$ the sequence of boundary conditions, then $\bar h_i\to 0$.

\end{theo}

Böhm also proved the following result on the behaviour of the distance between the two singular orbits (see \cite{bohm}, Section 6).

\begin{prop}
\label{convti}
    For all $i\in \nbN$, $\ti\leqslant \pi$ and 
    $$\ti \underset{i\to +\infty}{\longrightarrow }\pi.$$
\end{prop}

Finally, the asymptotic behaviour of a sequence $(g_i)$ is known. Before stating the convergence theorem, we introduce the cone metric, which arises as the limit of the sequence. 

\begin{defi}[Cone metric]
\label{defcone}
The cone metric is defined by

$$g_c(t)=dt^2+f_c(t)^2g^k+h_c(t)^2g^l,$$
where
$$f_c(t)=\sqrt{\frac{k-1}{k+l-1}}\sin (t),\qquad h_c(t)=\sqrt{\frac{l-1}{k+l-1}}\sin(t),\quad \quad \text{for } t\in [0,\pi].$$
    
\end{defi}

\rem{$g_c$ is singular at the singular orbits. Moreover, it is Einstein with Einstein constant $\Lambda=k+l$.}

Böhm proved that a sequence $(g_i)$ converges to $g_c$ in the following sense.

\begin{theo}[Theorem 3.7, \cite{bohm}]
    \label{convtocone}
    Let $(g_i)$ be a sequence of Einstein metrics on $S^{k+1}\times S^l$ (as in Theorem \ref{existprod}) or on $S^{k+l+1}$ (as in Theorem \ref{existsph}). Then, $$g_i\underset{i\to +\infty }{\longrightarrow} g_c$$ in the Gromov--Hausdorff distance.
\end{theo}

\section{The Lichnerowicz Laplacian and stability of Einstein metrics}

\label{secLichLap}

This Laplacian was introduced by A. Lichnerowicz in 1961 \cite{lich}. It appears naturally when one linearises Einstein's equations and studies gravitational waves. It is one of the Laplace-type operators that admit a Weitzenböck decomposition (see \cite{besse}, Chapter 1 for more details).

Let $M$ be a smooth, compact and oriented manifold. We denote by $\X(M)$ the space of smooth vector fields on $M$, by $\T(M)$ the space of symmetric $(0,2)$-tensors on $M$ and by $\TT(M)\subset \T(M)$ the space of transverse-traceless (TT) symmetric $(0,2)$-tensors on $M$.

Before defining the Lichnerowicz Laplacian, we define a curvature term that appears in the expression for this operator.

\begin{defi}

Let $T$ be a $(0,2)$-tensor on $M$ and $(e_A)$ an orthonormal basis of the tangent space at a point. We define$$\mathring R (T)(X,Y)=\sum_A T(R(X,e_A)Y,e_A).$$

\end{defi}

For more details, see \cite{besse}, 1.131b.

\begin{defi}[Lichnerowicz Laplacian]
\label{deflich}
The Lichnerowicz Laplacian is the linear operator $\Delta_L:\T(M)\to \T(M)$ defined by $$\forall T \in \T(M),~~\Delta_L T=\Blap T+\Rgot T,$$
where $\Blap=\nabla^*\nabla$ is the Bochner Laplacian, and 

$$\forall T\in \T(M),~\forall X,Y\in \X(M),~ \Rgot(T)(X,Y)=T(\Ric(X),Y)+T(X,\Ric(Y))-2\mathring R (T)(X,Y).$$

\end{defi}

We have the following property (see \cite{besse} 1.174d).

\begin{prop}
\label{lichbyric}
Let $T\in \TT(M)$. Then

$$d\Ric_g (T)=\frac{1}{2}\Delta_L(T).$$

\end{prop}

Finally, when restricted to $\TT(M)$, $\Delta_L$ has the standard properties: it is self-adjoint, elliptic and bounded below (see \cite{kroncke2016variational} for more details). From now on, we consider that the domain of $\Delta_L$ is $\TT(M)$.

In the case of Böhm metrics, we obtain an explicit expression for $\Delta_L$ on a particular class of TT tensors: $G$-invariant TT tensors, where $G$ denotes $SO(k+1)\times SO(l+1)$. We first describe this set.

\begin{prop}[$G$-invariant tensors]
    Let $g=dt^2+f(t)^2g^k+h(t)^2g^l$ be a Böhm metric on $S^{k+1}\times S^l$ or $S^{k+l+1}$, and $T$ a $G$-invariant tensor. Then, there exist three smooth functions $a,\phi,\psi$ such that $$T=a(t)dt^2+\phi(t)f(t)^2g^k+\psi(t)h(t)^2g^l.$$
 Conversely, such $T$ is $G$-invariant.
\end{prop}

\begin{proof}
    The existence of the functions $a,\phi,\psi$ on principal orbits follows from the fact that $SO(n+1)$-invariant symmetric $(0,2)$-tensors on $S^n$ are multiples of the round metric. The smoothness at the singular orbits follows from \cite{eschenburg2000initial}, Lemmas 1.1 and 1.2. Finally, the converse is immediate.
\end{proof}

   The set of $G$-invariant TT tensors is denoted by $\TtG(M)\subset \TT(M)$.

\rem{In the case of Böhm metrics, the TT conditions can be written     \begin{align*}
        a(t)+k\phi(t)+l\psi(t)&=0\qquad  \qquad\text{(traceless),}
        \\a'(t)+H(t)a(t)-k\frac{f'}f\phi(t) -l\frac{h'}{h}\psi(t)&=0\qquad  \qquad \text{(transverse),}
    \end{align*}
where $\displaystyle H(t):=k\frac{f'}{f}+l\frac{h'}{h}$ (the mean curvature).
}

Let $g$ be a Böhm metric on $S^{k+1}\times S^l$ or $S^{k+l+1}$.

Before computing $\Delta_L$ for Böhm metrics, we define a basis of the tangent space at a point of $[0,\t]\times S^k \times S^l$ as follows.

$$e_0=\partial_t,$$

$$\forall 1\leqslant i\leqslant k,~e_i=\frac{1}{f}E_i^k \quad \text{(where } (E_i^k) \text{ is a local geodesic orthonormal frame on } S^k \text{)},$$

$$\forall k+1\leqslant \alpha\leqslant d-1,~e_\alpha=\frac{1}{h}E_{\alpha-k}^l \quad \text{(where }(E_{\alpha-k}^l)\text{ is a local geodesic orthonormal frame on } S^l\text{)}.$$

Then, we can give the expression of $\Delta_L$ in this basis.

\begin{theo}[Explicit expression for the Lichnerowicz Laplacian]
\label{lichlap}
In the basis $(e_0,e_i,e_\alpha)$, the Lichnerowicz Laplacian of the tensor $T=a(t)dt^2+\phi(t)f(t)^2g^k+\psi(t)h(t)^2g^l\in \TtG(M)$ (where $M=S^{k+1}\times S^l$ or $S^{k+l+1}$) is diagonal, with
\begin{align*}
  (\Delta_L T)_{00}&=-a''(t)-Ha'(t)+2k\left(\frac{f'^2}{f^2}-\frac{f''}{f}\right)(a(t)-\phi(t))+2l\left(\frac{h'^2}{h^2}-\frac{h''}{h}\right)(a(t)-\psi(t)),
\\
  (\Delta_L T)_{ii}&=-\phi''(t)-H\phi'(t)-2\left(\frac{f'^2}{f^2}-\frac{f''}{f}\right)(a(t)-\phi(t))+2\frac{f'h'}{fh}l(\psi(t)-\phi(t)),
\\
  (\Delta_L T)_{\alpha\alpha}&=-\psi''(t)-H\psi'(t)-2\left(\frac{h'^2}{h^2}-\frac{h''}{h}\right)(a(t)-\psi(t))+2\frac{f'h'}{fh}k(\phi(t)-\psi(t)),
\end{align*}
where $\displaystyle H=k\frac{f'}{f}+l\frac{h'}{h}$.
    
\end{theo}

\begin{proof}
    The expression follows from a direct calculation using the definition and the Einstein equations. The main steps of this computation are presented in Appendix \ref{annexcalcul}. It is also possible to obtain the same expression using Formula 1.174 in \cite{besse}.

\end{proof}

We now give two different notions of stability (see also \cite{kroncke2021spectra}, Definition 1.2), and explain their relationship with the Lichnerowicz Laplacian.

\paragraph{Geometric stability.} This is the most classical notion of stability. It is related to the fact that Einstein metrics are critical points of the Einstein--Hilbert action $\Scal$, restricted to metrics of some fixed volume. See Definition 4.63 in \cite{besse} and \cite{koiso1979second} for more details.

\begin{defi}[Stable Einstein metric]
    Let $g$ be an Einstein metric. $g$ is called \textit{stable} if ${\Scal''_g}_{|\TT(M)}$ is non-positive.
\end{defi}

 By Proposition \ref{lichbyric}, this definition is related to the Lichnerowicz Laplacian.

\begin{prop}[A stability criterion]
\label{critstab}
Let $g$ be an Einstein metric with an Einstein constant $\Lambda$. Then $g$ is stable if and only if $$\Delta_L\geqslant 2\Lambda.$$
\end{prop}

\begin{proof}

    It is an immediate consequence of the following formula (see \cite{kroncke2016variational})
    $$\left.\frac{d^2}{dt^2}\right|_{t=0}\Scal(g+th)=-\frac{1}{2}\int_M\langle h,\Delta_Lh-2\Lambda h\rangle dV_g.$$
    
\end{proof}

\rem{Since $\Delta_L$ is elliptic and self-adjoint on a smooth compact manifold, its spectrum is discrete and bounded below with finite multiplicities. Hence the number of unstable directions is finite.}

\paragraph{Physical stability.} This notion is related to theoretical physics. Given an Einstein metric $g$, one can associate a generalised spacetime with this metric and study its stability (see \cite{gibbons2002gravitational}).

We first recall the definition of a generalised Schwarzschild--Tangherlini black hole (see \cite{lichbohm}, Section 2.2 for more details).

\begin{defi}
    Let $(M,g)$ be an $d$-dimensional Einstein manifold with an Einstein constant $d-1$.
    A generalised Schwarzschild--Tangherlini black hole associated with a $g$ is a metric of the form
     $$\hat g=\left[1-\left(\frac{\ell}{r}\right)^{d-1}\right]d\tau^2 +\frac{dr^2}{1-\left(\frac{\ell}{r}\right)^{d-1}}+r^2g,$$
     where $\ell$ is a constant.
\end{defi}

This solution, found by Tangherlini in 1963 \cite{tangherlini1963schwarzschild}, generalises Schwarzschild's solutions to the Einstein's equations from general relativity to arbitrary dimension $D=d+2$.

 Gibbons and Hartnoll proved the following instability criterion of such a spacetime \cite{gibbons2002gravitational,lichbohm}.

\begin{theo}[\cite{gibbons2002gravitational}, Section 3.3.3]

\label{stabphy}
           A generalised Schwarzschild--Tangherlini spacetime associated with an Einstein metric $g$ is unstable if the smallest eigenvalue $\mu_{\min}$ of $\Delta_L$ satisfies
        $$\mu_{\min}<\mu_*:=4-\frac{(5-d)^2}{4},$$ where $d$ is the dimension of the Einstein manifold and $d-1$ is the Einstein constant.
\end{theo}

Although these notions of stability arise in different contexts, they are both governed by the spectrum of the Lichnerowicz Laplacian.

In our case, we study Böhm metrics on some low-dimensional spaces. The dimension $d$ of the manifold satisfies $5\leqslant d\leqslant9$, thus $\mu_*\geqslant 0$, and more precisely, $\mu_*=0$ when $d=9$. Moreover, for these metrics, $\Lambda>0$. Therefore, if we show that the Lichnerowicz Laplacian admits negative eigenvalues, we can conclude that Böhm metrics are geometrically and physically unstable.

Finally, we define the index of the Lichnerowicz Laplacian associated with a Böhm metric, which is related to the number of unstable directions of a metric by Proposition \ref{critstab}.

\begin{defi}[Index]
\label{defiindex}
    Let $g$ be a Riemannian metric on $M$ and $\Delta_L$ be the Lichnerowicz Laplacian associated with this metric.
    The index of $\Delta_L$ is the number of negative eigenvalues of $\Delta_L$ counted with multiplicity. It is denoted by $\nu$.
\end{defi}

\rem{The properties of $\Delta_L$ directly imply that the index is finite. In our case, since $\Lambda>0$, the index is less than the coindex of a Böhm metric, which is the number of negative eigenvalues of $\Delta_L-2\Lambda$. Since we are also interested in studying the physical stability of Böhm metrics, we prefer to study the index, as defined in Definition \ref{defiindex}, rather than the coindex.}

\section{Main results}

\label{secMain}

Our first result is a stronger form of convergence for Böhm metrics than the one established in Theorem \ref{convtocone}, away from the singular orbits.
\begin{theo}

\label{convbohmcompact}

    Let $(g_i)$ be a sequence of Böhm metrics on $S^{k+l+1}$ or $S^{k+1}\times S^l$. For every compact subset $K$ of $(0,\pi)$, and every $n\in \nbN$,

    $$\left\|f_i^{(n)}-f_c^{(n)}\right\|_{\infty,K}\underset{i\to +\infty}{\longrightarrow} 0\quad \text{and} \quad \left\|h_i^{(n)}-h_c^{(n)}\right\|_{\infty,K}\underset{i\to +\infty}{\longrightarrow} 0.$$
\end{theo}

\rem{Since $\ti\to \pi$, the restrictions of $f_i$ and $h_i$ to $K$ are well defined for $i$ sufficiently large.}

To the best of our knowledge, this refinement does not appear in the literature, although it is presumably known to experts.

Then, we prove that the quadratic form of $\Delta_L^c$ is negative definite on subspaces of arbitrarily large dimension. More precisely, we prove the following.

\begin{theo}
\label{indexcone}
    Let $M=S^{k+1}\times S^l$ or $S^{k+l+1}$. We denote by $\TT^0(M)$ the space of TT tensors that vanish in a neighborhood of the singular orbits.
    
    Let $Q_c:\TT^0(M)\to \nbR$ be defined by $$Q_c(T)=\langle\Delta_L^cT,T\rangle_c,$$
    where $\langle\cdot ,\cdot\rangle_c$ denotes the $L^2$ scalar product for the metric $g_c$.

    For every $N\in \nbN$, there exists an $N$-dimensional vector space $E_c^N\subset \TT^0(M)$ such that $$\forall T\in E_c^N\setminus \{0\},\quad Q_c(T)<0.$$
\end{theo}

Using these two theorems, we obtain the following result.

\begin{theo}
    \label{divindex}
    Let $(g_i)$ be a sequence of Böhm metrics on $S^{k+1}\times S^l$ or $S^{k+l+1}$. For all $i\in \nbN$, let $\Delta_L^i$ be the Lichnerowicz Laplacian for $g_i$. Let $\nu_i$ be the index of $\Delta_L^i$. Then,

    $$\nu_i\underset{i\to +\infty}{\longrightarrow}+\infty .$$
\end{theo}

Since the index $\nu_i$ is less than or equal to the coindex of the metric $g_i$, an immediate consequence of this theorem is the following.

\begin{coro}
    Let $(g_i)$ be a sequence of Böhm metrics on $S^{k+1}\times S^l$ or $S^{k+l+1}$. Then, the number of unstable directions of $g_i$ tends to $+\infty$ as $i\to +\infty$ (\textit{i.e.} as the metrics approach $g_c$).
\end{coro}

Moreover, since $\mu_*\geqslant0$ in dimensions between 5 and 9, an immediate consequence of Theorems \ref{stabphy} and \ref{divindex} is the following result, which is a partial answer to the conjecture of Gibbons, Hartnoll and Pope.

\begin{coro}
    All the generalised Schwarzschild--Tangherlini spacetimes constructed from Böhm metrics on $S^{k+l+1}$ and $S^{k+1}\times S^l$ are unstable, except possibly a finite number of them.
\end{coro}

\section{Convergence of the Böhm functions}

\label{secConv}

The goal of this section is to prove Theorem \ref{convbohmcompact}. First, the smoothness of the system implies that it suffices to consider the convergence of $Y_i=(f_i,f_i',h_i,h_i')$. Second, since the trajectories $Y_i([0,\ti])$ converge to the trajectory of the cone metric by \cite{bohm}, Theorem 5.7, it remains to show that there is no time shift between these trajectories.

We prove the following result.

\begin{theo}
\label{convfirstcomp}
    Let $K$ be a compact subset of $(0,\pi)$. Let $Y_i=(f_i,f_i',h_i,h_i')$ and $Y_c=(f_c,f_c',h_c,h_c')$. Then $$\|Y_c-Y_i\|_{\infty,K}\underset{i\to+\infty}{\longrightarrow} 0.$$
\end{theo}

The following lemma shows that it is sufficient to prove Theorem \ref{convfirstcomp}.

\begin{lem}
    Theorem \ref{convfirstcomp} implies Theorem \ref{convbohmcompact}.
\end{lem}

\begin{proof}
    Let 
    $$X(f,f',h,h')=\left(f',-\Lambda f+(k-1)\frac{1-f'^2}{f}-l\frac{f'h'}{h},h',-\Lambda h+(l-1)\frac{1-h'^2}{h}-k\frac{f'h'}{f}\right),$$
defined on $\Omega=\{(f,f',h,h')\in \nbR^4~|~f>0 \text{ and } h>0\}$.
    The Einstein equations are just 
    \begin{equation}
    \label{einsimp}
            Y'=X(Y).
    \end{equation}
The result follows from the fact that $X$ is Lipschitz on a compact neighborhood of $Y_c(K)$.

\end{proof}

Next, we record a lemma that will be needed below.
\begin{lem}
\label{LemODE}
    Let $X:\Omega\to \nbR^4$ be a $C^\infty$ function defined on an open subset $\Omega$ of $\nbR^4$. Let $s_-<0<s_+$ and 
    $Z_c:[s_-,s_+]\to \Omega$ be a solution of $Z'=X(Z)$.

Set $K=Z_c([s_-,s_+])$, which is a compact subset of $\Omega$. Let $\eta_0>0$ be such that $T:=\{x\in \Omega~|~d(x,K)\leqslant \eta_0\}$ is a compact subset of $\Omega$. Since $X$ is locally Lipschitz, it admits a Lipschitz constant $L$ on $T$.

Finally, let $P\in \nbR^4$ be such that $|P-Z_c(0)|<\eta_0 e^{-LS}$, where $S=\max(-s_-,s_+)$.

Then the maximal solution $Z_P$ of 

$$\begin{cases}
    Z'=X(Z)
    \\ Z(0)=P
\end{cases}$$
is defined on $[s_-,s_+]$ and \begin{equation}
\label{contODE}
    \sup_{[s_-,s_+]}|Z_P-Z_c|\leqslant |P-Z_c(0)|e^{LS}.
\end{equation}
\end{lem}

\begin{proof}
  Adapted from \cite{hirsch2013differential}, Section 17.4.
\end{proof}

We can now prove Theorem \ref{convfirstcomp}.

\begin{proof}[Proof of Theorem \ref{convfirstcomp}]
   
    First, we notice that a solution $Y$ to (\ref{einsimp}) is injective, because (\ref{ein1}) implies that the mean curvature $H$ is decreasing.

    By Theorem 5.7 and Section 7 in \cite{bohm}, there exists a point $P_i$ on the trajectory of $Y_i$ such that $P_i\to Y_c(\pi/4)$. Since $Y_i$ is injective, there exists a unique $\tau_i\in(0,\ti)$ such that $P_i=Y_i(\tau_i)$.

    We now prove that $\tau_i\to \frac{\pi}{4}$. Let $\varepsilon\in (0,\pi/8)$, and $J_\varepsilon=\left[\varepsilon-\frac{\pi}{4},\frac{3\pi}{4}-\varepsilon\right]$. We set $$Z_c(t)=Y_c\left(t+\frac{\pi}{4}\right).$$
    $K_\varepsilon:=Z_c(J_\varepsilon)$ is a compact subset of $\Omega$. Then, there exists $\eta_\varepsilon>0$ such that $T_\varepsilon=\{x\in \nbR^4~|~d(x,K_\varepsilon)\leqslant \eta_\varepsilon\}$ is a compact subset of $\Omega$. We denote by $L(\varepsilon)$ the Lipschitz constant of $X$ on $T_\varepsilon$. Let $S(\varepsilon)=\frac{3\pi}{4}-\varepsilon$.

Let $Z_i$ be the solution to 
$$\begin{cases}
    Z'=X(Z)
    \\Z(0)=P_i
\end{cases}.$$
By the Picard-Lindelöf theorem, $Z_i$ is defined on a maximal interval $(\alpha_i,\beta_i)$. Moreover, since the equation is autonomous, uniqueness implies that $Z_i=Y_i(\tau_i+\cdot)$. Since $Y_i$ is defined on $(-\tau_i,\ti-\tau_i)$ and the boundary conditions lie outside $\Omega$, we have
$$(-\tau_i,\ti-\tau_i)=(\alpha_i,\beta_i).$$
By definition, $P_i\to Y_c(\pi/4)$, so there exists $i_\varepsilon\in \nbN$ such that

$$\forall i\geqslant i_\varepsilon,~~ |P_i-Z_c(0)|<\eta_\varepsilon e^{-L(\varepsilon)S(\varepsilon)}.$$
By Lemma \ref{LemODE}, for $i\geqslant i_\varepsilon$, we have that $Z_i$ is defined on $J_\varepsilon$. This implies
$$J_\varepsilon\subset (-\tau_i,\ti-\tau_i).$$
By Proposition \ref{convti}, $\ti\leqslant \pi$. Hence $$\forall i\geqslant i_\varepsilon,~~\left|\tau_i-\frac{\pi}{4}\right|<\varepsilon.$$
Thus, $\tau_i\to \frac{\pi}{4}$.

Finally, the uniform convergence follows from (\ref{contODE}) and the fact that $Z_c$ is $C^1$. 

\end{proof}

\rem{The choice of $\pi/4$ is arbitrary.}

\section{Cone metric}

\label{secConemetric}

In this section, we study the cone metric and establish Theorem \ref{indexcone}.  $M$ denotes $S^{k+1}\times S^l$ or $S^{k+l+1}$.

Fix $N\in \nbN$. To prove Theorem \ref{indexcone}, we need to find $N$ linearly independent TT tensors $T_c^1,\cdots,T_c^N$ such that $Q_c$ is negative definite on $\Span(T_c^1,\cdots,T_c^N)$. If the supports are pairwise disjoint, the independence is immediate and it is sufficient to prove that $Q_c(T_c^n)<0$ for all $1\leqslant n\leqslant N$. Thus, the goal of this section is to construct such tensors. First, we define a family of $G$-invariant TT tensors that are characterised by a scalar function. This makes the study of the quadratic form much easier. Then, we compute the quadratic form $Q_c$ (Equation (\ref{expressionQ})), and simplify this expression by making some assumptions on the tensors, to prove a simple criterion that ensures that the quadratic form is negative on some tensors (Lemma \ref{lemequiv}).

We first construct an injection from $C^\infty([0,\pi],\nbR)\cap\{\supp \subset (0,\pi) \}$ to $\TtG(M)\cap \TT^0(M)$, in order to define explicit $G$-invariant TT tensors.

\begin{prop}
\label{TTcone}
    We define, for all $\phi\in C^\infty([0,\pi],\nbR)$ such that $\supp(\phi)\subset (0,\pi)$, $$T(\phi):=\phi(t)f_c(t)^2g^k-\frac{k}{l}\phi(t)h_c(t)^2g^l.$$
Then $T$ is an injection from $C^\infty([0,\pi],\nbR)\cap \{\supp\subset (0,\pi)\}$ to $\TtG(M)\cap \TT^0(M)$.
\end{prop}

\begin{proof}
    Let $\phi\in C^\infty([0,\pi],\nbR)$. $T(\phi)$ is clearly $G$-invariant and traceless. The transverse property follows from $$\frac{f'_c}{f_c}=\frac{h_c'}{h_c}=\cot.$$
Finally, the injectivity follows directly from the expression of $T$.
\end{proof}

Let $Q(\phi):=Q_c(T(\phi)).$
It follows from Theorem \ref{lichlap} and Definition \ref{defcone} that if we set $W(t)=\frac{2}{\sin(t)^2}-2(k+l)\cot ^2(t)$, we have
\begin{align*}
    \Delta_L^cT_{ii}=-\phi''-H\phi'+W\phi \quad  \text{ and } \quad \Delta_L^cT_{\alpha\alpha}=-\psi''-H\psi'+W\psi,
\end{align*}
where $\psi=-(k/l)\phi$.
Thus,  up to a positive constant,
\begin{equation}
\label{expressionQ}
Q(\phi)=\int_0^\pi\sin^\Lambda(t)\phi'^2-2(\Lambda-1)\sin^{\Lambda-2}(t)\phi^2+2\Lambda\sin^\Lambda(t)\phi^2dt.
\end{equation} We also denote by $Q$ the extension of this quadratic form to $H_0^1((0,\pi))$.

To prove Theorem \ref{indexcone}, we need to find linearly independent functions $\phi_1,\cdots ,\phi_N$ such that $Q(\phi_n)<0$ for all $1\leqslant n\leqslant N$. First, we simplify the expression of $Q$ when the support of $\phi$ is close to 0. More precisely, we prove the following result.

\begin{lem}
    \label{lemqtilde}
Let $\eta$ be a smooth function such that $\operatorname{supp}(\eta)\subset (0,1)$.

Let $\phi_\delta(t)=\eta\left(\frac{t}{\delta}\right)$.
    Let $$\tilde Q(\phi)=\int_0^\pi t^\Lambda\phi'^2-2(\Lambda-1)t^{\Lambda-2}\phi^2dt.$$
    If $\delta$ is sufficiently small, and $\tilde Q(\phi_\delta)<0$, then $Q(\phi_\delta)<0$.
\end{lem}

\begin{proof}
    It follows directly from the Taylor expansion of $\sin$. 
\end{proof}

Let $\displaystyle\beta=\frac{\Lambda-1}{2}$. We make the following substitution

$$\phi(t)=t^{-\beta}v(\ln(t)).$$

The following criterion follows directly from this substitution.

\begin{lem}
\label{lemequiv}
    Let $\eta$ and $\delta>0$ be as in Lemma \ref{lemqtilde}. Let $v_\delta(s)=\phi_\delta(e^s)e^{\beta s}$. We define for all $v\in H^1(\nbR)$ $$S(v)=\frac{\displaystyle\int_\nbR v'^2}{\displaystyle\int_\nbR v^2}.$$
    Then, $$\tilde Q(\phi_\delta)<0 \quad \text{if and only if} \quad S(v_\delta)<\alpha^2,$$
 where $\alpha^2=2(\Lambda-1)-\beta^2=\frac{(9-\Lambda)(\Lambda-1)}4>0$ since $4\leqslant \Lambda \leqslant 8$.
\end{lem}

\begin{proof}
    A direct calculation leads to
$$\tilde Q(\phi_\delta)=\int_{-\infty}^{\ln(\delta)}v'_\delta(s)^2+\beta^2 v_\delta(s)^2ds-2(\Lambda-1)\int_{-\infty}^{\ln(\delta)}v_\delta(s)^2ds.$$
\end{proof}

The functional $S$ measures the frequency of $v_\delta$. This naturally suggests considering $v_\delta$ of the following form
$$v_\delta(s)=\delta^\beta\sin\left(\frac{\pi(s-s_0)}{L}\right)\mathds{1}_{[s_0,s_0+L]}.$$
$v_\delta$ is $H^1$, thus
\begin{align*}
    S(v_\delta)&=\frac{\displaystyle \int_{s_0}^{s_0+L}\frac{\pi^2}{L^2}\cos^2\left(\frac{\pi(s-s_0)}{L}\right)ds}{\displaystyle \int_{s_0}^{s_0+L}\sin^2\left(\frac{\pi(s-s_0)}{L}\right)ds}=\frac{\pi^2}{L^2}.
\end{align*}
We want $S(v_\delta)<\alpha^2$, therefore we have to choose $L>\frac{\pi}{\alpha}$.

We can adjust the parameter $s_0$, so that $\supp(\phi_\delta)\subset (0,\delta)$ and we can choose $N$ such functions with disjoint supports. More precisely, for all $1\leqslant n\leqslant N$, we define

$$\tilde \phi_n(t)=\left(\frac{t}{\delta}\right)^{-\beta}\sin\left(\frac{\pi}{L}\ln\left(\frac{t}{\delta}\right)\right)\mathds 1 _{[\exp(-2nL),\exp(-(2n-1)L)]}\left(\frac{t}{\delta}\right)=\eta\left(\frac{t}{\delta}\right).$$
$\Delta_L$ is a second-order operator, and $\tilde \phi_n$ is only $H^1$. Thus, we consider a regularization $\phi_n $ of $\tilde \phi_n$. Since $Q$ is continuous by (\ref{expressionQ}), it satisfies $Q(\phi_n)<0$ (if $\phi_n$ is sufficiently close to $\tilde \phi_n$). We can also assume that the supports of $\phi_n$ remain disjoint, because the distances between the supports of $\tilde \phi_n$ are positive. By Lemmas \ref{lemqtilde} and \ref{lemequiv}, $Q(\phi_n)<0$, and since the supports are disjoint, the family $(\phi_n)$ is orthogonal with respect to the bilinear form associated with $Q$. Thus, Theorem \ref{indexcone} follows by taking $$E_c^N=\Span(T_c^1,\cdots ,T^N_c),$$ where $T_c^n=T(\phi_n)$. We note that $E_c^N$ is $N$-dimensional by the injectivity of $T$.

\rem{In the next section, we will consider $\tilde T_c^n=T(\alpha_n\phi_{n+1}+\beta_n \phi_n)$ for some real constants $\alpha_n$ and $\beta_n$ instead of $T_c^n$. 

We note that $\tilde E_c^N:=\Span(\tilde T _c^2,\tilde T_c^4,\cdots ,\tilde T_c^{2N})$ satisfies the same properties as $E_c^N$.}

\section{Böhm metrics close to the cone metric}

\label{secClose}

In this section, we prove Theorem \ref{divindex} by contradiction. Thus, we assume that there exist an increasing function $\eta:\nbN\to \nbN$ and a constant $C>0$ such that 
\begin{equation}
\label{contraindex}
    \forall i\in \nbN,\quad  \nu_{\eta(i)}\leqslant C.
\end{equation}

In the previous section, we proved that we can find vector spaces of arbitrarily large dimension on which $Q_c$ is negative definite. Since Böhm metrics converge to $g_c$, we aim to prove that, for all $N\in \nbN$, for every $g_i$ sufficiently close to $g_c$, we can find an $N$-dimensional subspace of $\TT((M,g_i))$ on which the quadratic form $Q_i$ is negative definite.

Since the Böhm functions $f_i$ and $h_i$ are not known explicitly, it is more difficult to construct $G$-invariant TT tensors than in the cone case. More precisely, replacing $g_c$ by $g_i$ in the expression given in Proposition \ref{TTcone} does not give a TT tensor since, for general Böhm metrics, the relation $f_i'/f_i=h_i'/h_i$ does not hold. To overcome this problem, we introduce a correction term, originally defined in \cite{lichbohm}. Then, we prove that this family of tensors converges to the TT tensors defined for the cone metric (Lemma \ref{lemconvTi}). Next, we prove the convergence of the quadratic form (Proposition \ref{propconvquad}). Finally, Theorem \ref{divindex} follows from these results and the min-max theorem.

We first introduce the analogue of the map $T$ defined in Proposition \ref{TTcone} for Böhm metrics. Let $(g_i)$ be a sequence of Böhm metrics on $M$. Let $$W_i(t)=f_i(t)^kh_i(t)^{l+1}\quad \text{and}\quad D_i(t)=W_i(t)\frac{k}{l}\left(\frac{h'_i}{h_i}-\frac{f'_i}{f_i}\right).$$
For all $\gamma_i\in C^\infty ([0,\ti],\nbR)$ such that $\supp(\gamma_i)\subset (0,\ti)$, we define $$\chi_i(\gamma_i)(t)=\frac{1}{W_i(t)}\int_0^tD_i(s)\gamma_i(s)ds.$$

\rem{$\chi_i$ is the correction term introduced in \cite{lichbohm}, (62).}
Let   
$$T_i(\gamma_i)=-l\chi_i(\gamma_i)(t)dt^2+\gamma_i(t)f_i(t)^2g^k+\left(\chi_i(\gamma_i)(t)-\frac{k}{l}\gamma_i(t)\right)h_i(t)^2g^l,$$
and $\displaystyle J_i(\gamma_i)=\int_0^{\ti}D_i(s)\gamma_i(s)ds$.

\begin{prop}
\label{propTi}

$T_i(\gamma_i)$ is well defined for all $\gamma_i\in \ker (J_i)$, and $$\supp(T_i(\gamma_i))\subset [\inf(\supp(\gamma_i)),\sup(\supp(\gamma_i))].$$
Moreover, it is $G$-invariant and TT for the metric $g_i$. 

\end{prop}

\begin{proof}
    Let $\gamma_i\in \ker(J_i)$ with $\supp(\gamma_i)\subset (0,\ti)$. $T_i(\gamma_i)$ is clearly defined for all $t\in (0,\ti)$. Since $W_i(t)\to 0$ as $t\to 0$ or $\ti$, we need to check that $\chi_i(\gamma_i)$ is well defined at the singular orbits. We have $$\lim_{t\to 0^+}\chi_i(\gamma_i)(t)=\lim_{t\to 0^+}\frac{D_i(t)\gamma_i(t)}{W'_i(t)}=-\frac{1}{l}\gamma_i(0)=0.$$
Since $J_i(\gamma_i)=0$, we find after a similar calculation that $\chi_i(\gamma_i)(t)\to 0$ as $t\to \ti^-$.

The property of the support is immediate. Then, $T_i(\gamma_i)$ is clearly $G$-invariant and traceless. It is also transverse because $\chi_i(\gamma_i)$ satisfies
$$\chi_i(\gamma_i)'+\frac{W_i'}{W_i}\chi_i(\gamma_i)=\frac{k}{l}\left(\frac{h_i'}{h_i}-\frac{f_i'}{f_i}\right)\gamma_i.$$
    
\end{proof}

The following lemma establishes the convergence of the tensors defined in the above proposition.

\begin{lem}
\label{lemconvTi}
Let $(\gamma_i)$ be a sequence of smooth functions such that, for all $i\in \nbN$, $\supp(\gamma_i)\subset (0,\ti)$ and $\gamma_i \in \ker (J_i)$. We assume that for all $p\in \nbN$, ${\gamma_i}^{(p)}$ converges uniformly to some $\gamma_\infty^{(p)}\in C^\infty([0,\pi],\nbR)$ such that $\supp(\gamma_\infty)\subset (0,\pi)$. Then for every compact set $K\subset (0,\pi)$ and $p\in \nbN$,
$$\|T_i(\gamma_i)^{(p)}-T(\gamma_\infty)^{(p)}\|_{\infty,K}\to 0.$$
    
\end{lem}

\begin{proof}
    It is sufficient to prove that $\chi_i(\gamma_i)\to 0$ uniformly on $K$, and similarly for its derivatives. By Theorem \ref{convbohmcompact}, for all $p\in \nbN$, $$W_i^{(p)}\underset{i\to +\infty}{\longrightarrow} \left[f_c^kh_c^{l+1}\right]^{(p)}\quad \text{and} \quad D_i^{(p)} \underset{i\to +\infty}{\longrightarrow}0$$ uniformly on $K$, because $f_c'/f_c=h_c'/h_c$. Moreover, $f_c(t)^kh_c(t)^{l+1}>0$ for all $t\in K$. Thus the denominator does not vanish. This proves the lemma.
\end{proof}

We now prove a continuity property of the quadratic form.

\begin{prop}
\label{propconvquad}
   Let $(\gamma_i)$ be a sequence of smooth functions such that, for all $i\in \nbN$, $\supp(\gamma_i)\subset (0,\ti)$ and $\gamma_i \in \ker (J_i)$. Let $Q_i=\langle \Delta_L^i\cdot ,\cdot \rangle_i$, where $\langle \cdot ,\cdot \rangle_i$ denotes the $L^2$ scalar product for the metric $g_i$. We assume that for all $p\in \nbN$, $\gamma_i^{(p)}$ converges uniformly to some $\gamma_\infty^{(p)}\in C^\infty([0,\pi],\nbR)$ such that $\supp(\gamma_\infty)\subset (0,\pi)$. We also assume that there exists a closed interval $K\subset (0,\pi)$ such that $$\bigcup_{i\in \nbN\cup \{\infty\}}\supp(\gamma_i)\subset K.$$
Then $$Q_i(T_i(\gamma_i))\underset{i\to +\infty}{\longrightarrow} Q_c(T(\gamma_\infty)).$$
\end{prop}

\begin{proof}
  By Proposition \ref{propTi} and convexity of $K$, 

$$\supp(T(\gamma_\infty))\cup\bigcup_{i\in \nbN }\supp(T_i(\gamma_i))\subset K.$$
For all $T=a(t)dt^2+\phi(t)f_i(t)^2g^k+\psi(t)h_i(t)^2g^l\in \TtG((M,g_i))$, \begin{align*}Q_i(T)&=\int_0^{\ti}[(\Delta_L^iT)_{00}(t) a(t)+k(\Delta_L^i T)_{jj}(t)\phi(t)+l(\Delta_L^iT)_{\alpha\alpha}(t)\psi(t)]f_i(t)^kh_i(t)^ldt\\&=:\int_0^{\ti}\Psi_i(T)(t)dt.\end{align*}
Similarly, $$Q_c(T)=\int_0^\pi\Psi_c(T)(t)dt.$$
Thus, $$Q_i(T_i(\gamma_i))=\int_K\Psi_i(T_i(\gamma_i))(t)dt\quad \text{and} \quad Q_c(T(\gamma_\infty))=\int_K \Psi_c(T(\gamma_\infty))(t)dt.$$
Theorems \ref{lichlap} and \ref{convbohmcompact} and Lemma \ref{lemconvTi} yield
$$\|\Psi_i(T_i(\gamma_i))-\Psi_c(T(\gamma_\infty))\|_{\infty, K}\to 0.$$
This implies the convergence of the quadratic form.

\end{proof}

Let $N=\lfloor C\rfloor+1>C$. We can now end the proof of Theorem \ref{divindex}.

\begin{proof}[Proof of Theorem \ref{divindex}]

Let $1\leqslant n\leqslant 2N$. We set $$\gamma_i^n=\begin{cases}
    \frac{\displaystyle J_i(\phi_n)\phi_{n+1}-J_i(\phi_{n+1})\phi_n}{\displaystyle\sqrt{J_i(\phi_n)^2+J_i(\phi_{n+1})^2}}\quad \text{if } (J_i(\phi_n),J_i(\phi_{n+1}))\neq (0,0),
    \\ 
    \\
\frac{1}{\sqrt 2}\phi_{n+1}+\frac{1}{\sqrt 2}\phi_n\quad \text{otherwise.}
\end{cases}$$
For all $i\in \nbN$, $\gamma_i^n$ lies in $\ker(J_i)$, and it can be written $\gamma_i^n =\alpha_i^n \phi_{n+1}+\beta_i^n\phi_n$, with $v_i^n :=(\alpha_i^n,\beta _i^n)\in S^1$. By compactness, there exists an increasing function $\psi_1:\nbN\to \nbN$ such that,
$$\forall 1\leqslant n\leqslant 2N,\quad v_{\eta\circ\psi_1(i)}^n\to v_\infty ^n=:(\alpha_\infty^n,\beta_\infty^n).$$
We set $$\gamma_c^n=\alpha_\infty ^n\phi_{n+1}+\beta_\infty ^n \phi_n.$$
According to the proof of Theorem \ref{indexcone}, we can replace $T_c^n$ by $\tilde T_c^n=T(\gamma_c^n)$ and consider $\tilde E_c^N=\Span(\tilde T_c^2,\tilde T_c^4,\cdots ,\tilde T_c^{2N})$ instead of $E_c^N$. We notice that $\tilde E_c^N$ satisfies the properties of Theorem \ref{indexcone}. Let $\psi=\eta\circ \psi_1$. Hence, for all $p\in \nbN$, $${\gamma_{\psi(i)}^n}^{(p)}\underset{i\to +\infty }{\longrightarrow}{\gamma_c^n}^{(p)}.$$
For a fixed $1\leqslant n\leqslant 2N$, if $\tilde K_n:=\supp(\phi_n)\cup \supp(\phi_{n+1})$, then $$\supp(\gamma_c^n)\cup \bigcup_{i\in \nbN}\supp\left( \gamma_{\psi(i)}^n\right)\subset \tilde K_n\subset \operatorname{Hull}(\tilde K_n)=:K_n\subset (0,\pi).$$ By Proposition \ref{propconvquad}, 

\begin{equation}
\label{convQ}
\forall 1\leqslant n \leqslant 2N,\quad Q_{\psi(i)}\left(T_{\psi(i)}\left(\gamma^n_{\psi(i)}\right)\right)\underset{i\to +\infty }{\longrightarrow} Q_c(T(\gamma_c^n))<0.
\end{equation}
Moreover, for all $n\neq m$, $$\supp\left(T_{\psi(i)}\left(\gamma^{2n}_{\psi(i)}\right)\right)\cap \supp\left(T_{\psi(i)}\left(\gamma^{2m}_{\psi(i)}\right)\right)=\emptyset$$ by construction of $\phi_1,\cdots ,\phi_{2N+1}$. 
Therefore, if we set $$E_{\psi(i)}^N=\Span\left(T_{\psi(i)}\left(\gamma^2_{\psi(i)}\right),\cdots, T_{\psi(i)}\left(\gamma^{2N}_{\psi(i)}\right)\right),$$ then $E_{\psi(i)}^N$ is an $N$-dimensional subspace of $\TT((M,g_i))$.
Moreover, by (\ref{convQ}) and since the supports of the generators are pairwise disjoint, there exists $i(N)\in \nbN$ such that $$Q_{\psi(i)}<0\text{ on }E_{\psi(i)}^N$$ for every $i \in \nbN$ such that $\psi(i)\geqslant i(N)$,

Let $i_0(N)\in \nbN$ be such that, for all $i\geqslant i_0(N)$, $\psi(i)\geqslant i(N)$.

Finally, since $\Delta_L^i$ is self-adjoint and has a compact resolvent (because it is elliptic on a smooth compact manifold), its spectrum has the following form

$$\operatorname{Sp}(\Delta_L^i)=\{\mu^i_1\leqslant \mu_2^i\leqslant \cdots \}.$$
The min-max theorem yields
$$\mu_n^i=\inf_{\substack{V\subset \TT((M,g_i))\\\dim V=n}}\max_{T\in V\setminus\{0\}}\frac{Q_i(T)}{\|T\|^2}.$$
Hence, for all $i\geqslant i_0(N)$, $$\mu_N^{\psi(i)}<0.$$
Therefore, for all $i\geqslant i_0(N)$,
$$\nu_{\psi(i)}\geqslant N>C,$$ which contradicts (\ref{contraindex}). This concludes the proof of Theorem \ref{divindex}.

\end{proof}

\printbibliography

\newpage

\appendix

\section{Calculation of the Lichnerowicz Laplacian}

\label{annexcalcul}

The computation of the Lichnerowicz Laplacian is not complicated but quite long, thus we present here the intermediate results leading to the final expressions, such as the computation of the connection, the Bochner Laplacian, the Riemann curvature tensor...

\rem{We adopt Besse's convention for the Riemann curvature tensor.}

In this section, we consider a Böhm metric $g=dt^2+f(t)^2g^k+h(t)^2g^l$ on $S^{k+1}\times S^l$ or $S^{k+l+1}$.

First, we recall that we define a basis of the tangent space as follows.

$$e_0=\partial_t,$$

$$\forall 1\leqslant i\leqslant k,~e_i=\frac{1}{f}E_i^k \quad \text{(where } (E_i^k) \text{ is a local geodesic orthonormal frame on } S^k \text{)},$$

$$\forall k+1\leqslant \alpha\leqslant d-1,~e_\alpha=\frac{1}{h}E_{\alpha-k}^l \quad \text{(where }(E_{\alpha-k}^l)\text{ is a local geodesic orthonormal frame on } S^l\text{)}.$$
To make the notation clearer, we denote by $i,j...$ an element of $\{ 1,2,\cdots ,k \}$, by $\alpha,\beta...$ an element of $\{ k+1,k+2,\cdots,d-1 \}$ and by $A,B...$ an element of $\{0,1,\cdots,d-1 \}$.

We compute the connection using the Koszul formula.

\begin{prop}
\label{connection}
    \begin{align*}
    \nabla_{e_0}e_A&=0,
\\\nabla_{e_p}e_i&=-\frac{f'}{f}\delta_{ip}e_0+\nabla^k_{e_p}e_i
 ,&\nabla_{e_p}e_0&=\frac{f'}{f}e_p
,&\nabla_{e_p}e_\alpha&=0,
\\\nabla_{e_\pi}e_i&=0,
&\nabla_{e_\pi}e_0&=\frac{h'}{h}e_\pi,
&\nabla_{e_\pi}e_\alpha&=-\frac{h'}{h}\delta_{\pi\alpha}e_0+\nabla^l_{e_\pi}e_\alpha.
\end{align*}
where $\nabla^n$ denotes the Levi-Civita connection on $(S^n,g^n)$.
\end{prop}

We can calculate $\Blap$ with the following formula

$$\Blap T=-\sum_A (\nabla_{e_A}\nabla_{e_A} T-\nabla_{\nabla_{e_A}e_A} T),$$
(see \cite{besse} 1.135 and 1.55).

\begin{prop}[Expression of $\Blap$]

$$\Blap T=-T'' -H T'+\Gamma(T),$$
where (in the basis $(e_A)$)

$$\Gamma(T)=\begin{pmatrix}
-2k\frac{f'^2}{f^2}\phi-2l\frac{h'^2}{h^2}\psi+2\left(k\frac{f'^2}{f^2}+l\frac{h'^2}{h^2}\right)a& 0 & 0
\\ 0 & -2\frac{f'^2}{f^2}(a-\phi)I_k &0
\\ 0 & 0 & -2\frac{h'^2}{h^2}(a-\psi)I_l
\end{pmatrix}.$$
    
\end{prop}
Now, we compute $\Rgot$.

We first compute the Riemann curvature tensor.

\begin{prop}
    \label{riemann}
$$R(e_i,e_j)e_p=\frac{f'^2-1}{f^2}(\delta_{jp}e_i-\delta_{ip}e_j),
\qquad R(e_i,e_j)e_\alpha=0,
\qquad R(e_i,e_j)\partial_t=0,
$$
$$
R(e_\alpha,e_\beta)e_\gamma=\frac{h'^2-1}{h^2}(\delta_{\beta\gamma}e_\alpha-\delta_{\alpha\gamma}e_\beta),
\qquad R(e_\alpha,e_\beta)e_i=0,
\qquad R(e_\alpha,e_\beta)\partial_t=0,
$$

$$R(e_i,e_\alpha)e_j=-\frac{f'h'}{fh}\delta_{ij}e_\alpha
,\qquad R(e_i,e_\alpha)e_\beta=\frac{f'h'}{fh}\delta_{\alpha\beta}e_i,
\qquad R(e_i,e_\alpha)\partial_t=0,
$$

$$R(\partial_t,\partial_t)=0,$$ 

$$R(e_i,\partial_t)e_j=-\frac{f''}{f}\delta_{ij}\partial_t,
\qquad R(e_i,\partial_t)\partial_t=\frac{f''}{f}e_i,
\qquad R(e_i,\partial_t)e_\alpha=0,$$

$$R(e_\alpha,\partial_t)e_i=0,
\qquad R(e_\alpha,\partial_t)\partial_t=\frac{h''}{h}e_\alpha,
\qquad R(e_\alpha,\partial_t)e_\beta=-\frac{h''}{h}\delta_{\alpha\beta}\partial_t.$$

\end{prop}

Thus, the expression of $\Rgot$ follows, since $$\mathring{R}( T) (X,Y)=\sum_A T(R(X,e_A)Y,e_A).$$

\begin{prop}
    In the basis $(e_A)$
    \begin{align*}
   \mathfrak{R}(T)&=\begin{pmatrix}
 2\Lambda & 0 & 0
\\ 0 & 2\frac{f''}{f}I_k & 0
\\ 0 & 0 & 2\frac{h''}{h}I_l
\end{pmatrix}a(t)+
\begin{pmatrix}
2k\frac{f''}{f} & 0 & 0
\\ 0 & \left[2\Lambda-2(k-1)\frac{1-f'^2}{f^2}\right]I_k & 0
\\ 0 & 0 & 2\frac{f'h'}{fh}k I_l
\end{pmatrix}\phi(t)
\\&+\begin{pmatrix}
2l\frac{h''}{h} & 0 & 0
\\ 0 & 2\frac{f'h'}{fh}l I_k & 0
\\ 0 & 0 & \left[2\Lambda-2(l-1)\frac{1-h'^2}{h^2}\right] I_l
\end{pmatrix}\psi(t).
\end{align*}
\end{prop}

Therefore, we can finally calculate the Lichnerowicz Laplacian.

\begin{align*}
  (\Delta_L T)_{00}&=-a''(t)-Ha'(t)+2\left(\Lambda+k\frac{f'^2}{f^2}+l\frac{h'^2}{h^2}\right)a(t)+2k\left(\frac{f''}{f}-\frac{f'^2}{f^2}\right)\phi(t)+2l\left(\frac{h''}{h}-\frac{h'^2}{h^2}\right)\psi(t),
\end{align*}

\begin{align*}
  (\Delta_L T)_{ii}&=-\phi''(t)-H\phi'(t)+\left(2\Lambda-2\frac{k-1}{f^2}+2k\frac{f'^2}{f^2}\right)\phi(t)+2\left(\frac{f''}{f}-\frac{f'^2}{f^2}\right)a(t)+2\frac{f'h'}{fh}l\psi(t),
\end{align*}

\begin{align*}
  (\Delta_L T)_{\alpha\alpha}&=-\psi''(t)-H\psi'(t)+\left(2\Lambda-2\frac{l-1}{h^2}+2l\frac{h'^2}{h^2}\right)\psi(t)+2\left(\frac{h''}{h}-\frac{h'^2}{h^2}\right)a(t)+2k\frac{f'h'}{fh}\phi(t).
\end{align*}

Since $g$ is Einstein, the previous expressions reduce, by (\ref{ein1}), (\ref{ein2}) and (\ref{ein3}), to

\begin{align*}
  (\Delta_L T)_{00}&=-a''(t)-Ha'(t)+2k\left(\frac{f'^2}{f^2}-\frac{f''}{f}\right)(a(t)-\phi(t))+2l\left(\frac{h'^2}{h^2}-\frac{h''}{h}\right)(a(t)-\psi(t)),
\end{align*}

\begin{align*}
  (\Delta_L T)_{ii}&=-\phi''(t)-H\phi'(t)-2\left(\frac{f'^2}{f^2}-\frac{f''}{f}\right)(a(t)-\phi(t))+2\frac{f'h'}{fh}l(\psi(t)-\phi(t)),
\end{align*}

\begin{align*}
  (\Delta_L T)_{\alpha\alpha}&=-\psi''(t)-H\psi'(t)-2\left(\frac{h'^2}{h^2}-\frac{h''}{h}\right)(a(t)-\psi(t))+2\frac{f'h'}{fh}k(\phi(t)-\psi(t)).
\end{align*}

\end{document}